\documentclass[11pt, reqno]{amsart}

\usepackage{amsmath, amsfonts, amssymb, amsbsy, bigstrut, graphicx, enumerate,  upref, longtable, comment, comment}

\usepackage{pstricks}

\usepackage[breaklinks]{hyperref}

\usepackage[numbers, sort&compress]{natbib}

\newtheorem{thm}{Theorem}[section]

\newtheorem{cor}[thm]{Corollary}
\newtheorem{prop}[thm]{Proposition}
\newtheorem{defn}[thm]{Definition}

\theoremstyle{definition}

\newcommand{\bbe}{\mathbf{e}}
\newcommand{\bbr}{\mathbf{R}}
\newcommand{\bbp}{\mathbf{P}}

\newcommand{\bbz}{\mathbf{Z}}

\newcommand{\ee}{\mathbb{E}}

\newcommand{\mf}{\mathcal{F}}

\newcommand{\pp}{\mathbb{P}}

\numberwithin{equation}{section}

\usepackage[utf8]{inputenc}
\usepackage{amsmath}
\usepackage{amsfonts}
\usepackage{bbm}
\usepackage{mathtools}
\usepackage{tikz}
\usepackage{amsthm}
\usepackage[english]{babel}
\usepackage{fancyhdr}
\usepackage{amssymb}
\usepackage{enumitem}
\usepackage{qtree}
\usepackage{mathrsfs}

\DeclareMathOperator{\Poiss}{Poiss}

\newtheorem{lemma}[thm]{Lemma}

\newcommand{\1}{\mathbbm{1}}

\renewcommand{\bar}{\overline}

\renewcommand{\hat}{\widehat}

\newcommand{\cvp}{\xrightarrow{\mathbb{P}}}
\newcommand{\op}{\hat{p}}
\newcommand{\oT}{\overline{T}}

\begin{document}

\title{A phase transition for repeated averages}
\author{Sourav Chatterjee}
\author{Persi Diaconis}
\address{Departments of Mathematics and Statistics, Stanford University, USA}
\email{souravc@stanford.edu}
\email{diaconis@math.stanford.edu}
\author{Allan Sly}
\author{Lingfu Zhang}
\address{Departments of Mathematics, Princeton University, USA}
\email{lingfuz@math.princeton.edu}
\email{allansly@princeton.edu}

\thanks{Sourav Chatterjee's research was partially supported by NSF grant DMS-1855484}
\thanks{Persi Diaconis's research was partially supported by NSF grant DMS-0804324}
\thanks{Allan Sly's research was partially supported by DMS-1855527, Simons Investigator Grant and MacArthur Fellowship}

\begin{abstract}
Let $x_1,\ldots,x_n$ be a fixed sequence of real numbers. At each stage, pick two indices $I$ and $J$ uniformly at random and replace $x_I$, $x_J$ by $(x_I+x_J)/2$, $(x_I+x_J)/2$. Clearly all the coordinates converge to $(x_1+\cdots+x_n)/n$. We determine the rate of convergence, establishing a sharp `cutoff'  transition answering a question of Jean Bourgain. 
\end{abstract}

\maketitle

%\tableofcontents

\section{Introduction}\label{introsec}
Around 1980, Jean Bourgain asked one of us (personal communication to P.~D.) the question in the abstract. It recently resurfaced via a question in quantum computing (thanks to Ramis Movassagh). We record some convergence theorems. 

Fix $x_0 = (x_{0,1}, \ldots,x_{0,n})\in \bbr^n$ and define a Markov chain as follows: given $x_k$, pick two distinct coordinates $I$ and $J$ uniformly at random, and replace both $x_{k,I}$ and $x_{k,J}$ by $(x_{k,I}+x_{k,J})/2$, keeping all other coordinates the same, to obtain $x_{k+1}$. 

Let $\bar{x}_0 = \frac{1}{n}\sum_{i=1}^n x_{0,i}$. We study the rate of convergence of $x_k$ to the vector $(\bar{x}_0,\ldots,\bar{x}_0)$.  The expected $L^2$ norm can be computed exactly.  In Section \ref{proofsec} we show that
\begin{equation}\label{eq:L2bound}
\ee(\sum_{i=1}^n (x_{k,i}-\bar{x}_0)^2) = (1-\frac{1}{n-1})^k \sum_{i=1}^n (x_{0,i}-\bar{x}_0)^2
\end{equation}
giving convergence in $L^2$ for $k\gg n$. Convergence in $L^1$, however, happens in time of order $n\log n$.  The majority of the paper is devoted to the question of establishing that this transition occurs with cutoff and determining its location and window.

We will mostly focus on the initial condition $x_0=(1,0,\ldots,0)$ where all of the mass is concentrated on a single entry. This is a worst case as will be seen below.
We denote the $L^1$ distance from convergence by $T(k) = \sum_{i=1}^n |x_{k,i}-\bar{x}_0|$.
An immediate consequence of equation~\eqref{eq:L2bound} is that for $k=n \log n + cn$ with $c>0$ we have that
\[
\ee(T(k))\le e^{-c/2},
\]
giving $L^1$ convergence shortly after $n\log n$.
On the other hand, a simple counting argument shows that for $k=(\frac12-\epsilon) n\log n$ there are only $o(n)$ number of non-zero entries and hence $T(k)=2-o(1)$.

It is natural to ask whether either $\frac12 n\log n$ or $n\log n$ gives the cutoff location. In fact we establish that cutoff occurs strictly in between and that the leading order constant is $\frac1{2\log 2}$. Below and for the rest of this paper we use $\cvp$ to denote convergence in probability.
\begin{thm}\label{thm:main}
For $x_0=(1,0,\ldots,0)$, as $n\to\infty$ we have
\[
T\left(\theta n\log n\right) \cvp 2
\]
for any $\theta<\frac1{2\log 2}$, and
\[
T\left(\theta n\log n\right) \cvp 0,
\]
for any $\theta > \frac1{2\log 2}$.
\end{thm}
As we will explain in Section~\ref{s:sketch} the time $\frac1{2}n\log_2 n$ corresponds to the time at which a size biased co-ordinate of $x_k$ is $O(1/n)$.  Given that there is a sharp transition, it is natural to ask, how sharp?  In the following theorem we show that the window is order $n\sqrt{\log n}$ and characterise the descent of $T(k)$ in terms of the cumulative distribution function of the standard normal.

\begin{thm}\label{cutoffthm}
Let $\Phi:\bbr \to [0, 1]$ be the cumulative distribution function of the standard normal distribution $\mathcal{N}(0,1)$.
For $x_0=(1,0,\ldots,0)$, and any $a \in \bbr$, as $n\to \infty$ we have
$T(\lfloor n(\log_2(n) + a\sqrt{\log_2(n)})/2 \rfloor) \cvp 2\Phi(-a)$.
\end{thm}

\begin{figure}[t]
\centering
\includegraphics[width = .7\textwidth]{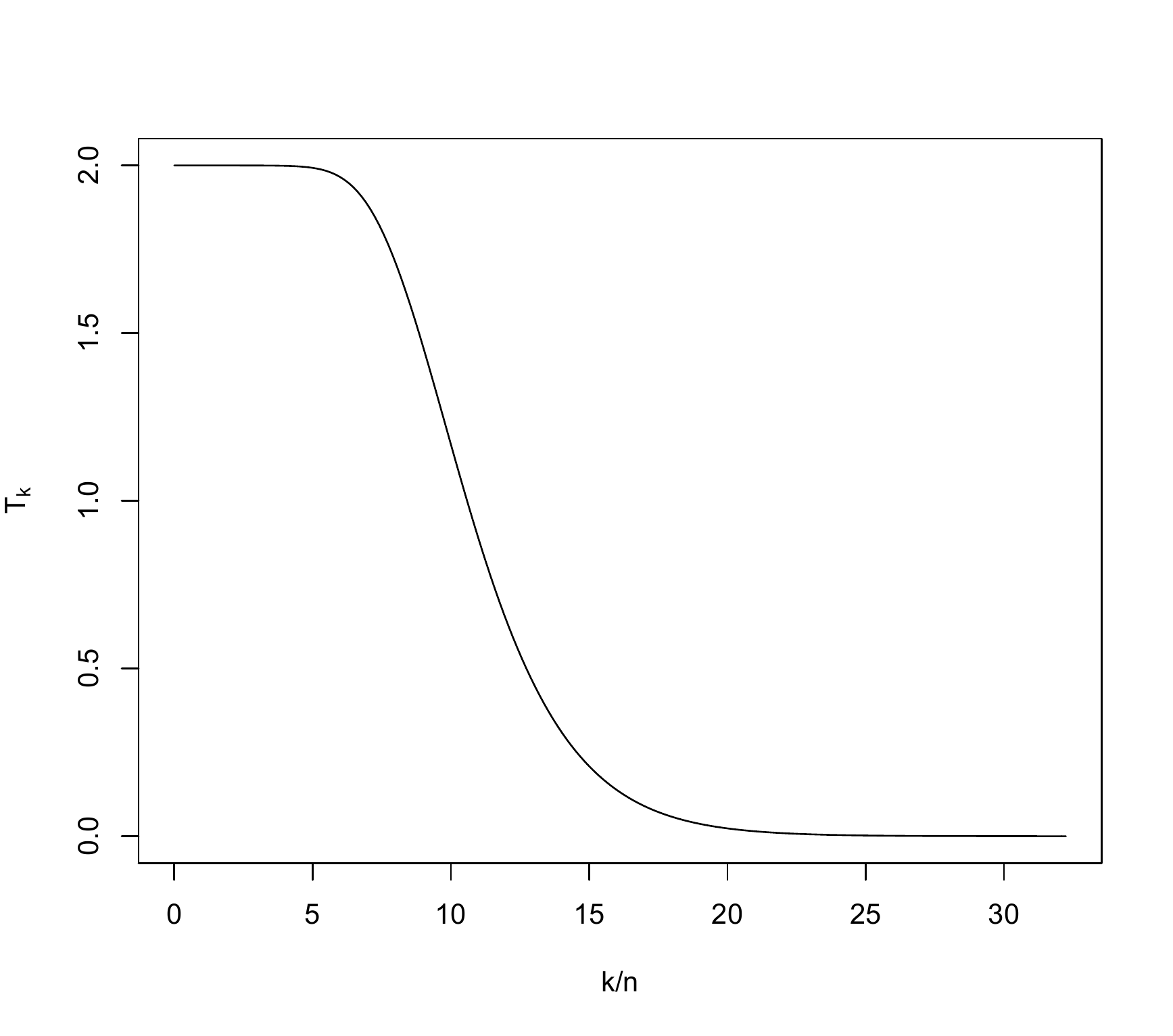}
\caption{Graph of $T(k)$ against $k/n$, with $n=10^7$ and $x_0=(1,0,\ldots,0)$.
\label{averagesfig1}}
\end{figure}
In a first version of this paper, we proved that the mixing time was between $\frac{1}{2}n\log n$ and $n\log n$. To try to determine things we ran a simulation with $n=10^7$. Figure \ref{averagesfig1} shows the result. It is not clear from these numerics if there is a cut off or not. Our results show the cutoff window is unusually large, making it very difficult to see from simulation.

One may also consider different general initial conditions, say, being non-negative with the same total mass; i.e., for initial conditions in the simplex 
\[\bbp_n=\{x_0 = (x_{0,1}, \ldots,x_{0,n}): \min_{1\le i\le n}x_{0,i}\ge 0,\;\; \sum_{i=1}^n x_{0,i}=1 \}.\]
The decay of the $L^1$ distance for some initial conditions could be much faster if initially the mass is more equality distributed, see e.g. Figure \ref{averagesfig2}.
Also, the decay does not necessarily have a cutoff. For example, if the initial condition has a quarter of the coordinates equal to $\frac{2}{n}$, one coordinate equals to $\frac{1}{2}$, and the remaining coordinates are zero, we would expect that there are `two sharp cutoffs', around the beginning and time $\frac{n\log n}{2\log 2}$ respectively.
However, it is natural ask if the initial condition of all mass concentrated on one entry is the worst case. We confirm this by proving the following asymptotic upper bound in probability.
\begin{thm}\label{generalinitthm}
We use $T_{x_0}(k)$ to denote $T(k)$ with initial condition $x_0$. For any $\epsilon>0$, as $n\to\infty$ we have
\[\sup_{x_0\in\bbp_n}\pp[T_{x_0}(\lfloor n(\log_2(n) + a\sqrt{\log_2(n)})/2 \rfloor) > 2\Phi(-a)+\epsilon] \to 0.\]
\end{thm}

Finally one might wonder if $x_{k}$ is ultimately exactly equal to the constant vector $(\bar{x}_0, \bar{x}_0,\ldots,\bar{x}_0)$. We show that for every choice of $x_0$ if and only if $n$ is a power of~$2$. 

\subsection*{Overview of this paper} Section \ref{backsec} contains a literature review pointing out occurrences of  repeated averaging processes in economics, game theory, actuarial science and mathematics. An outline of the proof is in Section \ref{s:sketch}. The proof itself is in Section \ref{proofsec} which gives additional results-in $L^2$, the upper bound under general initial conditions, and a necessary and sufficient condition for repeated averages to become constant is a finite number of steps.

\subsection{Background}\label{backsec}
Bourgain asked this question because of the second author's previous work on the random transpositions Markov chain. This evolves on the symmetric group $S_n$ by repeatedly picking $I$, $J$ uniformly at random and transposing these two labels in the current permutation. In joint work with Shahshahani~\cite{ds81}, we showed $\frac{1}{2}n\log n$ steps are necessary and sufficient for convergence to the uniform distribution in both $L^1$ and $L^2$. Map the symmetric group into the set of $n\times n$ doubly stochastic matrices by sending $(i,j)$ to the matrix 
\begin{equation*}
m_{ab} =
\begin{cases}
\frac{1}{2} &\text{ if } (a,b) = (i,i), (j,j), (i,j), (j,i),\\
1 &\text{ if } a=b\ne i \text{ or } j,\\
0 &\text{ otherwise.}
\end{cases}
\end{equation*}
The successive images of the random transformations are exactly our averaging operators.

It seemed difficult to transform the results for the random transpositions walk into useful results for random averages. It is worth noting that very sharp refinements have recently been proved for transpositions; sharp numerical bounds like
\[
\|Q^{*k} - U\|_{\textup{TV}}\le 2e^{-c} \ \ \text{ for } k = \frac{1}{2}n(\log n + c)
\]
are in \cite{scz} (where $Q^{*k}$ is the law of walk after $k$ steps, $U$ is the uniform distribution on $S_n$, and $\textup{TV}$ is the total variation norm) and even the limiting shape of the error is now understood~\cite{teyssier}. Good results for random $k$-cycles~\cite{bsz} suggest results for averaging over larger random sets which should be accessible with present techniques. Finally, the `shape' of the non-randomness for random transpositions if shuffling only $O(n)$ steps is of current interest because of its connection to spatial random permutations and the `exchange process' of mathematical physics (see~\cite{schramm, berestycki}). 

The $L^2$ convergence of a more general version of our repeated averaging process was studied by Aldous and Lanoue~\cite{al12} a few years ago. Aldous and Lanoue worked on an edge weighted graph with numbers at the vertices. At each step, an edge is picked with probability proportional to its weight and the two numbers on the vertices of the edge are replaced by their average. The `random transpositions' version of this process  was treated in~\cite{dsc2}.

Related processes, under the name of {\it gossip algorithms}, have been studied by Shah~\cite{shah08}. Such processes are also known as {\it distributed consensus algorithms}~\cite{ot09}. The {\it Deffuant model} from the sociology literature is a closely related model where averaging takes place only if the two values differ by less than a specified threshold~\cite{lanchier12, bkr03, haggstrom12}. Acemo\u{g}lu et al.~\cite{acemoglu} analyze a model where some agents have fixed opinions and other agents update according to an averaging process.

When $x_0$ lies in the positive orthant, the behavior of $T(k)$ has an amusing interpretation in terms of a toy model of reduction in wealth inequality in a socialist regime. Suppose that there are $n$ individuals (or entities) in the population, and the coordinates of $x_k$ denote the wealths of these $n$ individuals at time $k$. The socialist regime tries to redistribute wealth by picking two individuals uniformly at random at each time point, and making them equally distribute their wealths among themselves. This is our repeated averaging process. It is not hard to show that at time $k$, $\frac{1}{2}T(k)$ is the amount of wealth that remains to be redistributed to attain perfect equality. This is actually a well-known measure of wealth inequality in the economics literature, known as the Hoover index~\cite{hoover36} or Schutz index~\cite{schutz51}. 

What our results show is that if we start from the initial configuration where one individual has all the wealth, then for a long time this index of inequality does not decrease to any appreciable degree, and then starts decreasing gradually to zero. On the other hand, if we start from a wealth distribution where the wealth of the wealthiest individual is comparable to the average wealth, then the Hoover index decreases much faster, as shown by the following calculation. Suppose that the total wealth is $1$ (so that the average wealth is $1/n$), and the maximum wealth is $C/n$ for some $C\ge 1$.  Then by the results stated before,
\begin{align*}
\ee(T(k)) &\le \ee\biggl[\biggl(n\sum_{i=1}^n (x_{k,i}-\bar{x}_0)^2\biggr)^{1/2}\biggr]\\
&\le \biggl[n\ee\biggl(\sum_{i=1}^n (x_{k,i}-\bar{x}_0)^2\biggr)\biggr]^{1/2}\\
&\le \sqrt{n}\biggl(1-\frac{1}{n-1}\biggr)^{k/2} \biggl(\sum_{i=1}^n (x_{0,i}-\bar{x}_0)^2\biggr)^{1/2}\\
&\le \sqrt{n}e^{-k/2n} \biggl( \frac{C}{n} \sum_{i=1}^n |x_{0,i}-\bar{x}_0|\biggr)^{1/2}\le \sqrt{C} e^{-k/2n}. 
\end{align*}
A numerical example of this second scenario is shown in Figure \ref{averagesfig2}.

\begin{figure}[t]
\centering
\includegraphics[width = .7\textwidth]{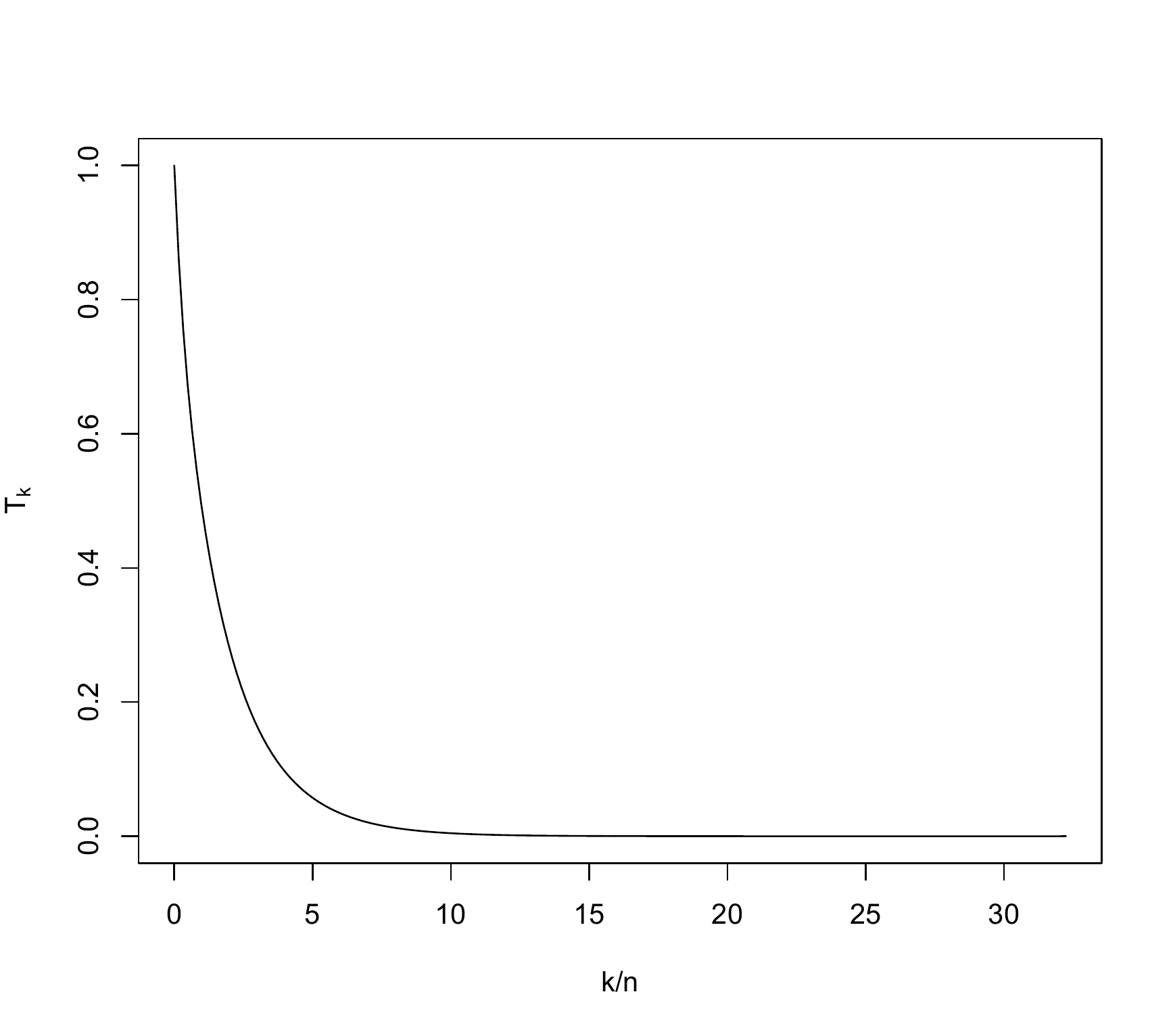}
\caption{Graph of $T(k)$ against $k/n$, with $n=10^7$ and half the coordinates of $x_0$ equal to $2/n$ and the remaining half equal to $0$. \label{averagesfig2}}
\end{figure}

Iterated local averages have a long tradition in the actuarial literature going back to Charles Peirce and de Forest. See \cite{dsc} for a survey. Replacing `averages of $3$' by `median of $3$' gives the `3RSSH smoother'. William Feller~\cite[p.~333, p.~425]{feller68} studies repeated averages for examples of the renewal theorem and Markov chains. An interesting literature on getting experts to reach consensus is surveyed and developed by Chatterjee and Seneta~\cite{cs}. Finally, our work can be set in the space of random walk on the semigroup of doubly stochastic matrices~\cite{hm11}. This subject does not seem to focus on the rates of convergence. The present paper suggests there is much to do. 

One mathematical use of iterated averages appears in summability theory~\cite{hardy}. Let $x_1,x_2,\ldots$ be a real sequence. Let $c^1_n(x) = \frac{1}{n}(x_1+\cdots +x_n)$ --- the first Cesaro average --- and let $c^{k+1}_n(x) = \frac{1}{n}(c^k_1(x)+\cdots c^k_n(x))$. Often $x_n$ is $1$ or $0$ as $n\in A$ or not, where $A\subseteq \{1,2,\ldots\}$. If $\lim c^1_n(x)$ exists this assigns a density to $A$. For $k\ge 1$, it can be shown that if $\{c^{k+1}_n(x)\}_{n=1}^\infty$ has a limit then $\{c^k_n(x)\}_{n=1}^\infty$ has a limit. However, $\liminf c_n^k(x)$ is increasing in $k$ and $\limsup c_n^k(x)$ is decreasing in $k$. If these meet as $k\to \infty$, $\{x_n\}_{n=1}^\infty$ is called $H_\infty$ summable. The sequence
\[
x_n =
\begin{cases}
1 &\text{ if lead digit of $n$ is $1$,}\\
0 &\text{ otherwise}
\end{cases}
\]
is not $c^k$ summable for any $k$ but has $H_\infty$ density $\log_{10}(2) = 0.301...$. For proofs and references see~\cite{diaconis}.

The repeated averaging process has an interesting connection with Schur convexity. The majorization partial order on $\bbr^n$ is defined as follows. Call $x = (x_1,\ldots,x_n) \preceq (y_1,\ldots, y_n) =y$ if $\sum_{i=1}^k x_{(i)} \le \sum_{i=1}^k y_{(i)}$ for all $1\le k\le n$, where $x_{(1)}\ge x_{(2)}\ge \cdots \ge x_{(n)}$ is a decreasing rearrangement of $x_1,\ldots,x_n$ and $y_{(1)}\ge y_{(2)}\ge \cdots \ge y_{(n)}$ is a decreasing rearrangement of $y_1,\ldots,y_n$. If all the entries are nonnegative and sum to $s$, the largest vector is $(s,0,\ldots,0)$ and the smallest vector is $(s/n,s/n,\ldots,s/n)$. An encyclopedic treatise on majorization is in \cite{moa11}.  A function $f:\bbr^n\to \bbr$ is called Schur convex if $x\preceq y$ implies $f(x)\le f(y)$.   It is easy to see that a symmetric convex function is Schur convex and that one moves down in the order by replacing $x_i$, $x_j$ by $(x_i+x_j)/2$, $(x_i+x_j)/2$. Therefore the relevance for the present paper is clear: each step of the Markov chain moves down in the order. Moreover, for any symmetric convex function $f$,  $f(x_{k+1})\le f(x_k)$ for all $k$. Thus, for instance, $\sum_{i=1}^n |x_{k,i}|^p$ is monotone decreasing in $k$ for any $p\ge 1$. 

The repeated averaging process has similarities with two familiar Markov chains. The first is the Kac walk --- a toy model for the Boltzmann equation that is widely studied in the physics and probability literatures. The walk proceeds on the unit sphere $\mathbf{S}^{n-1} = \{x\in \bbr^n: \sum_{i=1}^n x_i^2 = 1\}$. From $x\in \mathbf{S}^{n-1}$, choose distinct $I$ and $J$ uniformly at random and replace $x_I$ and $x_J$ by $x_I \cos \theta + x_J \sin \theta$ and $-x_I \sin\theta + x_J \cos \theta$, with $\theta$ chosen uniformly from $[0,2\pi)$. This is a surrogate for `random particles collide and exchange energy at random'. This Markov chain has a uniform stationary distribution on $\mathbf{S}^{n-1}$. Following a long series of improvements, the current best results on the rate of convergence of this walk are due to Pillai and Smith~\cite{ps17}. They show that order $n\log n$ steps are necessary and sufficient for mixing in total variation distance (indeed, $\frac{1}{2}n\log n$ is not enough and $200n\log n$ is enough). Aside from differences in state space and dynamics, the Kac walk has a uniform stationary distribution while repeated averaging is absorbing at a single point.

The second Markov chain that is similar to repeated averaging is the Gibbs sampler for the uniform distribution on the simplex $\Delta_{n-1} = \{x \in \bbr^n: x_i\ge 0 \ \forall i, \ x_1+\cdots +x_n =1\}$. 
The explicit description of this chain is as follows. From $x\in \Delta_{n-1}$, choose $I$ and $J$ uniformly at random and replace $x_I$, $x_J$ by $x_I'$, $x_J'$, with $x_I'$ chosen uniformly from $[0, x_I+x_J]$ and $x_J' = x_I+x_J - x_I'$. Settling a conjecture of Aldous, Aaron Smith~\cite{smith14} showed that order $n\log n$ steps are necessary and sufficient for convergence in total variation. Cutoff remains an open problem.

\subsection{Proof Sketch}\label{s:sketch}
The proof of the expected $L^2$ distance in equation~\eqref{eq:L2bound} is given by a direct computation in Proposition~\ref{explmm}.  One might wonder why this does not give the sharp $L^1$ upper bound.  The reason is that at time $(1-\epsilon)n\log_2 n$ a small fraction of the co-ordinates which have $o(1)$ of the total mass have large enough values that they give the dominant contribution to the $L^2$ norm while making a negligible contribution to the $L^1$ norm.  A similar phenomena happens when analysing the mixing time of random walks on random graphs such as random $d$-regular and Erdos-Renyi random graphs where standard spectral methods overestimate the mixing time.  It is from this analogy to random walks that our proof of Theorems~\ref{thm:main} and~\ref{cutoffthm} draws inspiration.

Let us explain this further in the case of random $d$-regular graphs where at time $\frac{d}{d-2}\log_{d-1}n$ the random walk has $L^1$ cutoff~\cite{LS:10}.  By the locally treelike nature of the graph, the walk can be coupled with a random walk on a $d$ regular tree up to time $(1-o(1))\frac{d}{d-2}\log_{d-1}n$.  This is the time at which the walk reaches the diameter of the graph $\log_{d-1}n$.  There is, however, a large deviation event that the walk does not move as far away from the root and that it only travels $(1-\delta)\log_{d-1}n$.  Since the distance from the starting location is a biased random walk, one can check that the probability of this event is roughly $e^{-c\delta^2\log n}$.  There are $n^{1-\delta}$ vertices at distance $(1-\delta)\log_{d-1}n$ so the transition probability to a such vertex $n^{-1+\delta-c\delta^2 +o(1)}$.  Their total contribution to the $L^2$ norm is thus $n^{1-\delta} (n^{\delta-c\delta^2 +o(1)})^2n^{-1}=n^{\delta-2c\delta^2+o(1)}$.  For small positive $\delta$ this diverges and so the contribution from this rare event blows up the $L^2$ norm.

The above observation suggest that one should study the walk conditional that it travels the typical distance from the origin.  This approach was introduced in~\cite{BLPS:18} to prove cutoff on the Erdos-Renyi random graph from almost all starting points. At the typical mixing time, it was shown that the distribution at time $(1-o(1))t_{mix}$ could be coupled to one with $L^\infty$ distance $n^{o(1)}$ from the stationary distribution.  This gives an $L^2$ bound of $n^{o(1)}$ and  using the spectral gap this can be reduced to $o(1)$ with a further time $o(\log n)$ establishing cutoff.

The repeated averaging process can be treated in a similar two step approach.  Just as the random walk on the random graphs initially can be coupled with a random walk on a tree, the repeated averaging process can initially be coupled with a fragmentation process where particles split exactly in two at rate $1/n$.  We will make this coupling even more explicit by Poissonizing the repeated averaging process, randomizing the number of particles.  The cutoff location is not when the number of particles becomes order $n$ but when most of the mass is in particles of size $n^{-1+o(1)}$.  This involves taking a size biased approach to the analysis.

We will now sketch how we carry out this analysis.  We imagine that the initial mass of 1 actually consists of a pile of $2^{H_n}$ individual particles each of mass $2^{-H_n}$ where $2^{H_n}=n^{1-o(1)}$.  When an averaging occurs we split the pile of particles in two.  When a single particle is to be split or when two non-zero piles are to be averaged we instead discard all these particles.  Hence a particle is in at most $H_n$ averagings.  At time $(\tfrac12-o(1))n\log_2 n$ we show that most particles are in piles of at most $n^{o(1)}$ by tracking how many averagings a given particle has participated in.  At this time particles in bigger piles are then also discarded and we show that only a negligible fraction of the particles are discarded.  

If we take the final set of particles we get a weight vector $w_t$ whose maximal value is bounded by $n^{-1+o(1)}$ and whose $L^1$ distance to $x_t$ is $o(1)$.  The argument is then completed by applying the $L^2$ analysis to the vector $w_t$ with its $L^2$ norm becoming $o(1/\sqrt{n})$  in additional time $o(\log n)$ giving an $L^1$ distance of $o(1)$ and establishing the upper bound of Theorem~\ref{thm:main}.  The lower bound on $T(k)$ can be recovered directly from the fragmentation process as before time $\tfrac12 n\log_2 n$ most of the mass is in piles of size much larger than~$n^{-1}$.

Our actual analysis is a little more complicated in order to establish the cutoff window and Theorem~\ref{cutoffthm}.  Essentially we split particles in groups according to when their pile becomes small and apply $L^2$ separately to each.  The number of averagings that a particle has participated in is Poisson and so its fluctuations are given by the Central Limit Theorem.  This explains the CDF in the statement of Theorem~\ref{cutoffthm}.

\subsection{Acknowledgments}\label{acksec}
We thank David Aldous, Laurent Miclo, Evita Nestoridi and Perla Sousi for comments. Our revival of this project stems from a quantum computing question of Ramis Movassagh --- roughly, to develop a convergence result with $\begin{bsmallmatrix}\frac{1}{2} & \frac{1}{2}\\ \frac{1}{2} &\frac{1}{2}\end{bsmallmatrix}$ replaced by the parallel quantum spin operator --- we hope to develop this in joint work with him. Finally, we acknowledge the great, late Jean Bourgain. He sent us a typed draft of a preliminary manuscript. Alas, 30 years later, this proved difficult to find. We will be grateful if a copy can be located. 

\section{Proofs}\label{proofsec}
Define a Markov chain $\{x_k\}_{k=0}^\infty$ as in Section \ref{introsec}. It is not difficult to see that without loss of generality, we can assume $\bar{x}_0=0$. We will work under this assumption throughout this section.
\subsection{Decay of expected $L^2$ distance}
For each $k$, let
\[
S(k) := \sum_{i=1}^n x_{k,i}^2. 
\]
Let $\mf_k$ be the $\sigma$-algebra generated by the history up to time $k$. 
\begin{prop}\label{explmm}
For any $k\ge 0$, 
\[
\ee(S(k+1)|\mf_{k}) =\biggl(1-\frac{1}{n-1}\biggr) S(k).
\]
\end{prop}
\begin{proof}
For any $i$, the probability that it is one of the chosen coordinates is $2/n$. If it is chosen, then the other coordinate is uniformly chosen among the remaining coordinates. Therefore
\begin{align*}
&\ee(x_{k+1,i}^2 |\mf_{k}) = \biggl(1-\frac{2}{n}\biggr) x_{k,i}^2 + \frac{2}{n(n-1)}\sum_{1\le j\le n, \, j\ne i}\biggl(\frac{x_{k,i}+x_{k,j}}{2}\biggr)^2\\
&=  \biggl(1-\frac{2}{n}\biggr) x_{k,i}^2 + \frac{1}{2n}x_{k,i}^2 + \frac{1}{2n(n-1)}\sum_{1\le j\le n, \, j\ne i}(x_{k,j}^2 + 2x_{k,i}x_{k,j}).
\end{align*}
Since $\sum_{i=1}^n x_{k,i}=0$, we have 
\[
\sum_{1\le j\le n, \, j\ne i} x_{k,j} = -x_{k,i}.
\]
Thus, we get the further simplification
\begin{align*}
&\ee(x_{k+1,i}^2 |\mf_{k}) \\
&=  \biggl(1-\frac{2}{n}+\frac{1}{2n}-\frac{1}{n(n-1)}\biggr)x_{k,i}^2 + \frac{1}{2n(n-1)}\sum_{1\le j\le n, \, j\ne i}x_{k,j}^2\\
&= \biggl(1-\frac{2}{n}+\frac{1}{2n}-\frac{1}{n(n-1)}-\frac{1}{2n(n-1)}\biggr)x_{k,i}^2 + \frac{1}{2n(n-1)}S(k)\\
&= \biggl(1-\frac{3}{2(n-1)}\biggr)x_{k,i}^2 + \frac{1}{2n(n-1)}S(k).
\end{align*}
Summing over $i$, we get the required identity.
\end{proof}
\begin{cor}\label{expcor}
Let $\tau := 1-\frac{1}{n-1}$. Then  $\ee(S(k)) = \tau^k S(0)$. Moreover, $\lim S(k)/\tau^k$ exists and is finite almost surely.
\end{cor}
\begin{proof}
The claim  $\ee(S(k)) = \tau^k S(0)$ is immediate by Proposition \ref{explmm} and induction. 
Proposition \ref{explmm} also shows that $M_k := S(k)/\tau^k$ is a nonnegative martingale, and so its limit  exists and is finite almost surely. 
\end{proof}

Incidentally, the argument works for more general methods of averaging. For example, if $x_I$ and $x_J$ are replaced by $\theta x_I + \bar{\theta} x_J$ and $\bar{\theta} x_I + \theta x_J$, where $0<\theta<1$ and $\bar{\theta} = 1-\theta$, Proposition \ref{explmm} becomes 
\[
\ee(S(k+1)|\mf_k) = \biggl(1-\frac{4\theta \bar{\theta}}{n-1}\biggr)S(k).
\]
Corollary \ref{expcor} shows that $S(k)$ is small when $k/n\gg 1$. 

\subsection{Cutoff in decay of $L^1$ distance}
Let us now investigate the decay of the $L^1$ norm of $x_k$. 
Throughout this subsection we assume that the starting state is $x_0=(1-\frac{1}{n},-\frac{1}{n},\ldots,-\frac{1}{n})$.
For each $k$, let 
\[
T(k) := \sum_{i=1}^n |x_{k,i}|.
\]
In Theorems~\ref{thm:main} and~\ref{cutoffthm} we claim that the  `cutoff phenomenon' holds for the decay of $T(k)$, around $k = \frac12 n\log_2(n)$ with cutoff window of order of $n\sqrt{\log(n)}$.  Note that $T(k)$ is decreasing in $k$ and so Theorem~\ref{cutoffthm} implies Theorem~\ref{thm:main}.

It will be more convenient for the proof to work in continuous time $\bbr_+$ instead and so we introduce a rate 1 Poisson clock on $\bbr_+$, so that when the clock rings at time $t$ we uniformly choose two different coordinates and replace them by their average.
This is equivalent to giving each pair of coordinates an independent Poisson clock with rate $\binom{n}{2}^{-1}$, for the times at which the pairs are averaged.
For any time $t\in \bbr_+$, we denote $x_t'=(x_{t,1}',\ldots, x_{t,n}')$, $T'(t)$ and $S'(t)$ to be the corresponding quantities for the continuous time model. Similarly to the discrete case, $T'(t)$ and $S'(t)$ are decreasing in $t$ and analogously to Proposition~\ref{explmm} we have that for $s<t$,
\begin{equation}\label{eq:contexplmm}
\ee(S'(t)|\mf_{s}) =\exp\biggl(-\frac{t-s}{n-1}\biggr)S'(s).
\end{equation}

The following theorem is the continuous time version of Theorem \ref{cutoffthm}.
\begin{thm}   \label{continuouscutoffthm}
Take $\Phi:\bbr \rightarrow [0, 1]$ as in Theorem \ref{cutoffthm}.
For any $a \in \bbr$, as $n\rightarrow \infty$ we have
$T'(n(\log_2(n) + a\sqrt{\log_2(n)})/2) \cvp 2\Phi(-a)$.
\end{thm}
It is not hard to see that Theorem \ref{cutoffthm} and Theorem \ref{continuouscutoffthm} are equivalent,
due to concentration of Poisson random variables, and the fact that $T(k)$ decreases in $k$.

For the convenience of notations, for the rest of this section we denote $t(a)=t(n,a):=n(\log_2(n) + a\sqrt{\log_2(n)})/2$.

Now we prove Theorem \ref{continuouscutoffthm}.
To analyze the evolution of $x_t'$, we couple it to a simpler repeated averaging process.
\begin{defn}
We define $w_t=(w_{t,1}, \ldots, w_{t,n})$, for $t\in \bbr_+$ coupled with the process $x_t'$ as follows.
We let $w_0=(1,0,\ldots, 0)$, and let it evolve by repeated averages using the the same Poisson point processes as $x_t'$.
In addition, at any time $t$, if $w_{t,i}$ and $w_{t,j}$ are chosen to average, with both $w_{t,i}, w_{t,j}>0$, we replace both $w_{t,i}$ and $w_{t,j}$ by zero.  Furthermore, let $H_n:=\lfloor\log_2(n)- (\log_2(n))^{1/3}\rfloor$, if $w_{t,i}=w_{t,j}<2^{-H_n}$ after the average, we also replace both of them by zero.
\end{defn}
Under this construction, by induction in time we always have that $w_{t,i}\leq x_{t,i}'+\frac{1}{n}$, and each nonzero $w_{t,i}$ is a (non-positive) integer power of $2$.
To analyze this $w_t$ process, we further define another process called the `particle model'.

\begin{defn}
Consider $2^{H_n}$ particles indexed by $u \in \{1, \ldots, 2^{H_n}\}$.
Let $\{e_{t,u}\}$ be the label of particle $u$ at time $t$ as they evolve over time, and take values in $\{0,1,\ldots,n\}$.  We set $e_{0,u}=1$ for all $u \in \{1, \ldots, 2^{H_n}\}$ as the initial values. Location 0 will correspond to a graveyard site for removed particles.
We further denote $P_{t,i}:=\{u: e_{t,u}=i\}$.
Then these sets $\{P_{t,i}\}_{t\in\bbr_+, 0\leq i \leq n}$ encode the same information as $\{e_{t,u}\}_{t\in\bbr_+, 1\leq u \leq 2^{H_n}}$.

Given the same Poisson clocks as in the repeated averaging process we define the evolution of the particles according to them,
to ensure that there is always $|P_{t,i}|=2^{H_n}w_{t,i}$.
Suppose at time $t$ the clock rings for edge $(i,j)$, then we apply a `splitting' in the particle model as follows.
If $\{|P_{t-,i}|, |P_{t-,j}|\}=\{2L, 0\}$, for some $L\in \bbz_+$, we uniformly randomly divide the particles such that half remain and the other half move to the other location giving $|P_{t,i}|=|P_{t,j}|=L$.

In all other cases, i.e. both $|P_{t-,i}|, |P_{t-,j}|$ are positive, or one of them is one,
the particles are discarded to location 0 and we have that $|P_{t,i}|=|P_{t,j}|=0$.
\end{defn}

By induction we always have that $|P_{t,i}|$ is power of $2$, so it is odd only when it equals $1$.  Moreover, the coupling with the repeated averaging process satisfies $|P_{t,i}|=2^{H_n}w_{t,i}$.  Using this particle model, we prove the following `weighted estimate' on the process $w_t$.

We define $p_{t,i} \in \bbz_{\geq 0}$ such that $w_{t,i}=2^{-p_{t,i}}$, if $w_{t,i}\neq 0$;
otherwise we let $p_{t,i} = \infty$.
For simplicity of notations we also let $\op_{t,u}=p_{t,e_{t,u}}$ if $e_{t,u}\neq 0$, and $\op_{t,u}=\infty$ otherwise.
\begin{prop}   \label{weightedestimate}
Take $a\in\bbr$ and $\delta \geq 0$, then we have as $n\rightarrow \infty$,
\[
\sum_{i=1}^n w_{t(a),i}\1[p_{t(a),i} \leq \log_2(n) - \delta\sqrt{\log_2(n)}] \cvp \Phi(-a-\delta).
\]
\end{prop}

We set up some further notations before the proof.

For each $1 \leq u \leq 2^{H_n}$ and $t\in\bbr_+$, let $\alpha_{t,u}\in\{0, 1, \ldots, H_n+1\}$ be the number of times $0<t'<t$, where the pair $(e_{t',u}, j)$ is chosen to average, for some $j\neq e_{t',u}$, $1 \leq j \leq n$, and $w_{t',j}=0$;
and let $\beta_{t,u}\in\{0, 1\}$ be the number of times $t'<t$, where the pair $(e_{t',u}, j)$ is chosen to average, for some $j\neq e_{t',u}$, $1 \leq j \leq n$, and $w_{t',j}>0$.

From the above definitions, if $\beta_{t,u}=1$ or $\alpha_{t,u}=H_n+1$, we have $e_{t,u}=0$; otherwise we have $e_{t,u}>0$, and $\op_{t,u}=\alpha_{t,u}$.

\begin{lemma}  \label{decayofbeta}
As $n\rightarrow \infty$, we have $\pp[\beta_{t(a),u}=1]\rightarrow 0$ uniformly for each $u\in\bbz_+$.
\end{lemma}

\begin{proof}
At any time $0<t'<t(a)$ we have $|\{j:w_{t',j} > 0\}| \leq 2^{H_n}$, so
\[
\pp[\beta_{t(a),u}=1] \leq 1-\exp\left(-2^{H_n} t(a)\binom{n}{2}^{-1}\right) < 2^{H_n} t(a)\binom{n}{2}^{-1}.
\]
By $2^{H_n}=2^{\lfloor\log_2(n)- (\log_2(n))^{1/3}\rfloor}\leq n 2^{-(\log_2(n))^{1/3}}$, and the definition of $t(a)$,
we have that $\lim_{n\rightarrow \infty}2^{H_n} t(a)\binom{n}{2}^{-1}=0$.
\end{proof}

\begin{proof}[Proof of Proposition \ref{weightedestimate}]
Using the particle model, and the definitions above, we have
\begin{align*}
&\sum_{i=1}^n w_{t(a),i}\1[p_{t(a),i} \leq \log_2(n) - \delta\sqrt{\log_2(n)}]\\
\quad=&
2^{-H_n}\sum_{u=1}^{2^{H_n}} \1[\op_{t(a),u} \leq \log_2(n) - \delta\sqrt{\log_2(n)}]\\
\quad=&
2^{-H_n}\sum_{u=1}^{2^{H_n}} \1[\beta_{t(a),u}=0]\1[\alpha_{t(a),u} \leq (\log_2(n) - \delta\sqrt{\log_2(n)})\wedge H_n].
\end{align*}
Here and below $\wedge$ denotes taking the minimum of two numbers.
By Lemma \ref{decayofbeta}, it suffices to show that 
\begin{equation}   \label{eq:alphaconverge}
2^{-H_n}\sum_{u=1}^{2^{H_n}} \1[\alpha_{t(a),u} \leq (\log_2(n) - \delta\sqrt{\log_2(n)})\wedge H_n] \cvp \Phi(-a-\delta).
\end{equation}
We will show that the expectation converges to the desired values, and the variance decays to zero.

We consider the distribution of $\alpha_{t(a),u}$ and $\alpha_{t(a),u'}$, for some $1\leq u < u' \leq 2^{H_n}$.
At any time $0 < t'<t(a)$, the number of empty sites satisfies $n-2^{H_n} \leq |\{j:w_{t',j} = 0\}| \leq n$.
This means that conditioned on $\beta_{t(a),u}=\beta_{t(a),u'}=0$, 
at any time $0 \leq t' \leq t(a)$, $\alpha_{t',u}$ and $\alpha_{t',u'}$ grow with rate between
$(n-2^{H_n})\binom{n}{2}^{-1}$ and $n\binom{n}{2}^{-1}$,
unless they equal $H_n+1$.

To study the joint law of $\alpha_{t(a),u}$ and $\alpha_{t(a),u'}$, we let
\[\gamma := \inf\{t \in \bbr_+: e_{t,u}\neq e_{t,u'}\}\cup \{t(a)\}.\]
Conditioned on $\beta_{t(a),u}=\beta_{t(a),u'}=0$, for any $\gamma < t' < t(a)$, we must have that $e_{t',u}\neq e_{t',u'}$, unless $\alpha_{t',u}=\alpha_{t',u'}=H_n+1$ and $e_{t',u}=e_{t',u'} =0$.
Thus conditioned on $\beta_{t(a),u}=\beta_{t(a),u'}=0$, we can couple the joint distribution of
$\alpha_{t(a),u}-\alpha_{\gamma,u}$, $\alpha_{t(a),u'}-\alpha_{\gamma,u'}$ with some $\underline{\alpha}$, $\underline{\alpha}'$,
such that
\begin{equation}   \label{eq:boundalphalower}
\underline{\alpha}\leq \alpha_{t(a),u}-\alpha_{\gamma,u} \;\text{and}\;
\underline{\alpha}'\leq \alpha_{t(a),u'}-\alpha_{\gamma,u'},
\end{equation}
where $\underline{\alpha}$ and $\underline{\alpha}'$ are independent $\Poiss\left((n-2^{H_n}) (t(a)-\gamma)\binom{n}{2}^{-1}\right)$ upper truncated by $H_n+1$.
We can also couple the joint distribution of
$\alpha_{t(a),u}-\alpha_{\gamma,u}$, $\alpha_{t(a),u'}-\alpha_{\gamma,u'}$ with some $\overline{\alpha}$, $\overline{\alpha}'$,
such that
\begin{equation}   \label{eq:boundalphaupper}
\overline{\alpha}\geq \alpha_{t(a),u}-\alpha_{\gamma,u} \;\text{and}\;
\overline{\alpha}'\geq \alpha_{t(a),u'}-\alpha_{\gamma,u'},
\end{equation}
where $\overline{\alpha}$ and $\overline{\alpha}'$ are independent $\Poiss\left(n (t(a)-\gamma)\binom{n}{2}^{-1}\right)$ upper truncated by $H_n+1$.

Next we control $\gamma$.
We show that as $n\rightarrow \infty$, $\pp[\gamma > n(\log_2(n))^{1/3}]\rightarrow 0$ uniformly in $u, u'$.
Indeed, at any time $t'$, if $e_{t',u}=e_{t',u'}$ is chosen to average with some $j$ such that $w_{t',j}=0$, then with probability $\geq \frac{1}{2}$, $e_{u}$ is going to be different from $e_{u'}$.
Thus $e_{u}$ and $e_{u'}$ become different at rate lower bounded by $\frac{1}{2}(n-2^{H_n})\binom{n}{2}^{-1}$, and this implies that $\pp[\gamma > n(\log_2(n))^{1/3}] \leq \exp\left(-\frac{1}{2}(n-2^{H_n})\binom{n}{2}^{-1}n(\log_2(n))^{1/3}\right)$, which decays to zero as $n\rightarrow \infty$.

At the same time, given $\gamma$, the law of $\alpha_{\gamma,u}=\alpha_{\gamma,u'}$ is stochastically dominated by $\Poiss\left(n\gamma \binom{n}{2}^{-1}\right)+1$.
Thus we have that 
\begin{equation}  \label{eq:decayoftc}
\frac{\alpha_{\gamma,u}}{\sqrt{\log_2(n)}} \cvp 0,
\end{equation}
uniformly in $u, u'$.

From the law of $\underline{\alpha}, \underline{\alpha}'$, and the estimate on $\gamma$, 
we have $\frac{\underline{\alpha}-\log_2(n)-a\sqrt{\log_2(n)}}{\sqrt{\log_2(n)}}$,
$\frac{\underline{\alpha}'-\log_2(n)-a\sqrt{\log_2(n)}}{\sqrt{\log_2(n)}}$ jointly converge in law to two independent standard Gaussian random variables, each upper truncated by $-a$;
and the same is true for  $\frac{\overline{\alpha}-\log_2(n)-a\sqrt{\log_2(n)}}{\sqrt{\log_2(n)}}$ 
and
$\frac{\overline{\alpha}'-\log_2(n)-a\sqrt{\log_2(n)}}{\sqrt{\log_2(n)}}$.
By \eqref{eq:boundalphalower}, \eqref{eq:boundalphaupper}, and \eqref{eq:decayoftc}, we have $\frac{\alpha_{t(a),u}-\log_2(n)-a\sqrt{\log_2(n)}}{\sqrt{\log_2(n)}}$ 
and
$\frac{\alpha_{t(a),u'}-\log_2(n)-a\sqrt{\log_2(n)}}{\sqrt{\log_2(n)}}$ also jointly converge in law to two independent standard Gaussian random variables, each upper truncated by $-a$.
Thus as $n\rightarrow \infty$, for $\delta > 0$, uniformly in $u, u'$ we have
\[
\pp[\alpha_{t(a),u} \leq \log_2(n)-\delta\sqrt{\log_2(n)}] \rightarrow \Phi(-a-\delta),
\]
\[
\pp[\alpha_{t(a),u'} \leq \log_2(n)-\delta\sqrt{\log_2(n)}] \rightarrow \Phi(-a-\delta),
\]
\[
\pp[\alpha_{t(a),u}, \alpha_{t(a),u'} \leq \log_2(n)-\delta\sqrt{\log_2(n)}] \rightarrow \Phi(-a-\delta)^2.
\]
For the case where $\delta=0$, we also have that
\[
\pp[\underline{\alpha}=H_n+1], \pp[\underline{\alpha}'=H_n+1], \pp[\overline{\alpha}=H_n+1], \pp[\overline{\alpha}'=H_n+1] \rightarrow 1-\Phi(-a),
\]
\[
\pp[\underline{\alpha}=\underline{\alpha}'=H_n+1], \pp[\overline{\alpha}=\overline{\alpha}'=H_n+1] \rightarrow (1-\Phi(-a))^2,
\]
so
\[
\pp[\alpha_{t(a),u}=H_n+1], \pp[\alpha_{t(a),u'}=H_n+1] \rightarrow 1-\Phi(-a),
\]
\[
\pp[\alpha_{t(a),u}=\alpha_{t(a),u'}=H_n+1] \rightarrow (1-\Phi(-a))^2.
\]
Thus for either $\delta > 0$ or $\delta = 0$, we have
\[
\pp[\alpha_{t(a),u} \leq (\log_2(n)-\delta\sqrt{\log_2(n)})\wedge H_n] \rightarrow \Phi(-a-\delta),
\]
\[
\pp[\alpha_{t(a),u}, \alpha_{t(a),u'} \leq (\log_2(n)-\delta\sqrt{\log_2(n)})\wedge H_n] \rightarrow \Phi(-a-\delta)^2,
\]
uniformly in $u, u'$.
By summing over $u$ and $u, u'$, we have
\[
\ee \left[2^{-H_n}\sum_{u=1}^{2^{H_n}} \1[\alpha_{t(a),u} \leq (\log_2(n)-\delta\sqrt{\log_2(n)})\wedge H_n]\right] \rightarrow \Phi(-a-\delta),
\]
\[
\ee \left[\left(2^{-H_n}\sum_{u=1}^{2^{H_n}} \1[\alpha_{t(a),u} \leq (\log_2(n)-\delta\sqrt{\log_2(n)})\wedge H_n]\right)^2\right] \rightarrow \Phi(-a-\delta)^2.
\]
These imply \eqref{eq:alphaconverge}, so our conclusion follows.
\end{proof}

From this we immediately get one side of Theorem \ref{continuouscutoffthm}. 

\begin{cor}   \label{easysidecutoff}
For any $\eta > 0$, we have
\[
\lim_{n\rightarrow \infty}\pp[T'(t(a)) > 2\Phi(-a)-\eta]=1.
\]
\end{cor}

\begin{proof}
From the construction of $w_t$, for any $1\leq i \leq n$ we have $x_{t(a),i}'+\frac{1}{n}\geq w_{t(a),i}$, and we have either $w_{t(a),i} = 0$, or $w_{t(a),i} \geq 2^{-H_n} > \frac{1}{n}$.
Thus we have
\begin{align*}
T'(t(a)) &=2\sum_{i=1}^n (x_{t(a),i}')^+\geq 2\sum_{i=1}^n w_{t(a), i}
 - \frac{2}{n}|\{1\leq i \leq n: w_{t(a),i} \neq 0\}|.
\end{align*}

Since $0 \leq |\{1\leq i \leq n: w_{t(a),i} \neq 0\}| \leq 2^{H_n}$, the second term converges to zero as $n\rightarrow \infty$.
For the first term, by Proposition \ref{weightedestimate} it converges to $\Phi(-a)$ in probability which completes the corollary.
\end{proof}

For the other side of Theorem \ref{continuouscutoffthm},
we need to bound $\sum_{i:w_{t(a),i}=0}|x_{t(a),i}|$.
To do this, we split the process $x_t$ into ones with fast-decaying $L_1$ norm.

\begin{defn}
Given $\delta>0$ and $m\in \bbz_+$.
Define a sequence of times $t[1], \ldots, t[m]$ as $t[k]:=t(a-\frac{\delta(m-k+4)}{2})$.
For each $1\leq u \leq 2^{H_n}$, we define
\[
k_u:=\min \{ 1 \leq k\leq m: \log_2(n)-\delta\sqrt{\log_2(n)} < \op_{t[k], u} < \infty\}\cup\{\infty\}.
\]

For each $1\leq k\leq m$, we define a repeated averaging process $\{x_t^{(k)}\}_{t\geq t[k]}$ as following.
We let
\[
x_{t[k], i}^{(k)}:=2^{-H_n}|\{1\leq u \leq 2^{H_n}: k_u=k,  e_{t[k],u}=i\}|,
\]
and since times $t[k]$, we let $x_t^{(k)}$ evolve using the same averaging pairs as $x_t'$.
We also denote $\bar{x}^{(k)}:=\frac{1}{n}\sum_{i=1}^n x_{t[k],i}^{(k)}$.
\end{defn}

From these definitions we have the following results.
\begin{lemma}  \label{boundpart}
We have
$x_{t[k], i}^{(k)} < \frac{2^{\delta\sqrt{\log_2(n)}}}{n}$.
\end{lemma}
\begin{proof}
From the definition above, if $p_{t[k], i} \leq \log_2(n)-\delta\sqrt{\log_2(n)}$,
then for any $u$ with $e_{t[k], u} = i$ we must have $k_u \neq k$, so we have $x_{t[k], i}^{(k)} = 0$.
Otherwise, since we have $x_{t[k], i}^{(k)}\leq 2^{-H_n}|P_{t[k], i}| = w_{t[k],i} = 2^{-p_{t[k],i}}$, our conclusion follows as well.
\end{proof}

\begin{lemma}
Almost surely we have
\begin{equation}   \label{eq:partboundstrong}
x_{t,i}'+\frac{1}{n} \geq \sum_{k:t[k]\leq t} x_{t,i}^{(k)}
+ 2^{-H_n} |\{1\leq u \leq 2^{H_n}: t[k_u]>t, e_{t,u}=i \}|,
\end{equation}
where we assume that $t[\infty]=\infty$.
\end{lemma}
\begin{proof}
Denote the right hand side by $Z_{t,i}$.  First note that at times $t[k]$ the right hand side does not change since almost surely there are no clock rings and the extra term $x_{t,i}^{(k)}$ in the sum is exactly compensated for by the decrease in the second term from particles with $k_u=k$.

We establish the lemma by considering the evolution of the averaging process.    If the clock rings for $(i,j)$ at time $t$ the pair $(x_{t-,i}'+\frac{1}{n},x_{t-,j}'+\frac{1}{n})$ is replaced with $\frac{x_{t-,i}'+x_{t-,j}'}{2}+\frac{1}{n}$ at $i$ and $j$.  The right hand side also undergoes an averaging but sometimes when particles collide or are singletons.
In any cases we have that
\[
Z_{t,i} \leq \frac{Z_{t-,i}+Z_{t-,j}}{2}\leq \frac{x_{t-,i}'+x_{t-,j}'}{2}+\frac{1}{n} = x_{t,i}'+\frac{1}{n}
\]
Thus by induction we have that \eqref{eq:partboundstrong} holds.
\end{proof}

\begin{lemma}  \label{lowerboundpart}
Given any $\eta \in \bbr_+$ and $a \in \bbr$, we can take $\delta$ small enough, then $m$ large enough, so that
$\pp\left[\sum_{k=1}^m\sum_{i=1}^n x_{t[k],i}^{(k)} < 1 - \Phi(-a) - \eta\right] \rightarrow 0$ as $n\rightarrow \infty$.
\end{lemma}

\begin{proof}
From the definition we have
\[
\sum_{k=1}^m\sum_{i=1}^n x_{t[k],i}^{(k)}
=
2^{-H_n}|\{1 \leq u \leq 2^{H_n}: k_u \neq \infty\}|.
\]
For any $1\leq u \leq 2^{H_n}$, $\op_{t,u}$ is non-decreasing in $t$.
Thus if $k_u = \infty$, we must have one of the following cases
\begin{itemize}
    \item Either $\op_{t[1],u}=\infty$, i.e., $e_{t[1],u}=0$;
    \item Or $\op_{t[m],u} \leq \log_2(n) - \delta\sqrt{\log_2(n)}$;
    \item Or for some $1\leq k \leq m-1$ we have that $\op_{t[k],u} \leq \log_2(n) - \delta\sqrt{\log_2(n)}$ and $\op_{t[k+1],u}=\infty$, i.e., $e_{t[k+1],u}=0$.
\end{itemize}

By Proposition \ref{weightedestimate}, 
\begin{equation} \label{eq:initterm}
   2^{-H_n}\sum_{u=1}^{2^{H_n}}\1[e_{t[1],u}=0] \cvp 1-\Phi\left(-a+\frac{\delta (m+3)}{2}\right),
\end{equation}
and
\begin{equation} \label{eq:remainterm}
2^{-H_n}\sum_{u=1}^{2^{H_n}}\1[\op_{t[m],u} \leq \log_2(n) - \delta\sqrt{\log_2(n)}] \cvp \Phi(-a+2\delta).
\end{equation}

Now we take $1\leq k \leq m-1$ and estimate
\begin{equation}  \label{eq:fastjumpterm}
2^{-H_n}\sum_{u=1}^{2^{H_n}}\1[\op_{t[k],u} \leq \log_2(n) - \delta\sqrt{\log_2(n)}, e_{t[k+1],u}=0].
\end{equation}
For each $u$ we have
\begin{align*}
& \1[\op_{t[k],u} \leq \log_2(n) - \delta\sqrt{\log_2(n)}, e_{t[k+1],u}=0 ]
\\
\leq &
\1[\beta_{t[k+1],u}=1]
\\
&+
\1[\beta_{t[k+1],u}=0, \alpha_{t[k],u}\leq \log_2(n) - \delta\sqrt{\log_2(n)}, \alpha_{t[k+1],u}=H_n+1].
\end{align*}
By Lemma \ref{decayofbeta} we have $\1[\beta_{t[k+1],u}=1]\rightarrow 0$ uniformly in $u$.
Also, conditioned on $\beta_{t[k+1],u}=0$, we have that  $\alpha_{t[k+1],u}-\alpha_{t[k],u}$ is dominated by $\Poiss\left(n(t[k+1]-t[k])\binom{n}{2}^{-1}\right)=\Poiss\left(\frac{n^2\sqrt{\log_2(n)}\delta}{4}\binom{n}{2}^{-1}\right)$.
This implies that
$\pp[\alpha_{t[k+1],u}-\alpha_{t[k],u}\geq \delta\sqrt{\log_2(n)} - (\log_2(n))^{1/3}]\rightarrow 0$ as $n\rightarrow \infty$, uniformly in $u$.
Thus \eqref{eq:fastjumpterm} decays to zero in probability as $n\rightarrow \infty$.

Summing over $1\leq k\leq m-1$, and using \eqref{eq:initterm} and \eqref{eq:remainterm}, for any $\epsilon > 0$ we have
\[
\pp\left[\sum_{k=1}^m\sum_{i=1}^n x_{t[k],i}^{(k)} < \Phi\left(-a+\frac{\delta (m+3)}{2}\right) - \Phi(-a+2\delta) - \epsilon \right] \rightarrow 0.
\]
By choosing $\epsilon < \eta$, taking $\delta$ small, then $m$ large, we have $\Phi\left(-a+\frac{\delta (m+3)}{2}\right) - \Phi(-a+2\delta) - \epsilon > 1 - \Phi(-a) - \eta$, and our conclusion follows.
\end{proof}

\begin{proof}[Proof of Theorem \ref{continuouscutoffthm}]
According to Corollary \ref{easysidecutoff}, it suffices to show that for any $\eta > 0$, we have $\lim_{n\rightarrow \infty}\pp[T'(t(a)) > 2\Phi(-a)+\eta]=0$.

We have
\begin{align*}
T'(t(a)) &= \sum_{i=1}^n |x_{t(a), i}'|
\\
&\leq \sum_{i=1}^n \left|x_{t(a), i}' + \sum_{k=1}^m \left(\bar{x}^{(k)} - x_{t(a),i}^{(k)}\right)\right| + \sum_{k=1}^m \left|x_{t(a),i}^{(k)} - \bar{x}^{(k)}\right|.
\end{align*}
For each $1\leq i \leq n$, we denote $r_i:=x_{t(a), i}' + \frac{1}{n} - \sum_{k=1}^m x_{t(a),i}^{(k)}$.
By \eqref{eq:partboundstrong} we have $r_i\ge 0$, and \[\bar{r}: = \frac{1}{n}\sum_{i=1}^n r_i = \frac{1}{n}\sum_{i=1}^n \left(x_{t(a), i}' + \frac{1}{n} - \sum_{k=1}^m x_{t(a),i}^{(k)}\right) = \frac{1}{n} - \sum_{k=1}^m \bar{x}^{(k)}.\]
Then we have
\begin{align*}
\sum_{i=1}^n \left|x_{t(a), i}' + \sum_{k=1}^m \left(\bar{x}^{(k)} - x_{t(a),i}^{(k)}\right)\right|
& =
\sum_{i=1}^n |r_i - \bar{r}| 
\\
& \leq 2 \sum_{i=1}^n r_i = 2 - 2\sum_{k=1}^m\sum_{i=1}^n x_{t(a),i}^{(k)}.
\end{align*}
By Lemma \ref{lowerboundpart}, if we take $\delta$ small enough then $m$ large enough, the right hand side is asymptotically bounded by $2\Phi(-a) + \frac{\eta}{2}$.

For any $1\leq k \leq m$, by the continuous time version of Proposition \ref{explmm}, \eqref{eq:contexplmm}, we have
\begin{align*}
&\sum_{i=1}^n \ee[|x_{t(a),i}^{(k)}-\bar{x}^{(k)}|\mid x_{t[k]}^{(k)}]
\\
\leq &\sqrt{n\sum_{i=1}^n \ee[(x_{t(a),i}^{(k)}-\bar{x}^{(k)})^2\mid x_{t[k]}^{(k)}]}
\\
=&
\sqrt{n \exp\left(-\frac{t(a)-t[k]}{n-1}\right)\sum_{i=1}^n (x_{t[k],i}^{(k)}-\bar{x}^{(k)})^2}
\\
\leq &
\sqrt{\exp\left(-\frac{\delta n\sqrt{\log_2(n)}}{n-1}\right)2^{\delta \sqrt{\log_2(n)}}\sum_{i=1}^n x_{t[k],i}^{(k)}},
\end{align*}
where in the last inequality, we used $t(a)-t[k] \geq t(a)-t[m] = \delta n\sqrt{\log_2(n)}$, and Lemma \ref{boundpart}.
Now summing over $k$ we get
\begin{align*}
&\sum_{k=1}^m\sum_{i=1}^n \ee[|x_{t(a),i}^{(k)}-\bar{x}^{(k)}|]
\\
\leq &
\sqrt{\exp\left(-\frac{\delta n\sqrt{\log_2(n)}}{n-1}\right)2^{\delta \sqrt{\log_2(n)}}}\sqrt{m\ee\left[\sum_{k=1}^m\sum_{i=1}^n x_{t[k],i}^{(k)}\right]}
\\
\leq &
\sqrt{\exp\left(-\frac{\delta n\sqrt{\log_2(n)}}{n-1}\right)2^{\delta \sqrt{\log_2(n)}}m}.
\end{align*}
Here we used $\sum_{k=1}^m\sum_{i=1}^n x_{t[k],i}^{(k)} = \sum_{k=1}^m\sum_{i=1}^n x_{t(a),i}^{(k)} \leq \sum_{i=1}^n x_{t(a), i}' + \frac{1}{n}=1$ in the last inequality.
For any fixed $\delta$ and $m$, as $n\rightarrow\infty$ the above converges to zero.
This implies that $\sum_{k=1}^m\sum_{i=1}^n |x_{t(a),i}^{(k)}-\bar{x}^{(k)}| \cvp 0$.
Thus asymptotically $T'(t(a))$ is bounded by $2\Phi(-a)+\eta$ in probability.
\end{proof}

\subsection{An upper bound on $L^1$ decay for general initial conditions}
We now prove Theorem \ref{generalinitthm} using Theorem \ref{cutoffthm}, and linearity of the repeated averaging process.

We work with all initial states $x_0\in\bbp_n$, i.e., $x_{0,i}\ge 0$ for each $i$ and $\sum_{i=1}^n x_{0,i}=1$.
Let $T_{x_0}(k)=\sum_{i=1}^n \left|x_{k,i}-\frac{1}{n}\right|$ be the $L^1$ distance starting from $x_0$.
For each $i=1,\ldots, n$ we denote $\bbe_i\in \bbr^n$ as the vector with $1$ at the $i$-th coordinate, and $0$ at all other coordinates. Then $\bbp_n$ is the convex hull of $\bbe_1,\ldots, \bbe_n$.
\begin{proof}[Proof of Theorem \ref{generalinitthm}]
We couple the repeated average processes starting from each $x_0\in \bbp_n$, so that all of them use the same pair of indices to average at each step.
Thus we get a coupling of all $T_{x_0}$. By linearity of the repeated average, we have that under this coupling, $x_k$ starting from $x_0$ is a linear combination of those $x_k$ starting from $\bbe_1,\ldots,\bbe_n$, with weights $x_{0,1},\ldots,x_{0,n}$ respectively. So we have
\[
T_{x_0}(k) \le \oT_{x_0}(k) := \sum_{i=1}^n x_{0,i}T_{\bbe_i}(k),
\]
for each $x_0\in \bbp_n$ and each $k$.
Using that $\sum_{i=1}^n x_{0,i}=1$, and that $\ee[T_{\bbe_i}(k)]$ and $\ee[(T_{\bbe_i}(k))^2]$ are independent of $i$, we have $\ee[\oT_{x_0}(k)] = \ee[T_{\bbe_1}(k)]$, and
\[
\ee[(\oT_{x_0}(k))^2]=\sum_{i,j=1}^n x_{0,i}x_{0,j}\ee[T_{\bbe_i}(k)T_{\bbe_j}(k)] \le \ee[(T_{\bbe_1}(k))^2].
\]
By Theorem \ref{cutoffthm}, as $n\to\infty$ we have
\[
\inf_{x_0\in \bbp_n} \ee[\oT_{x_0}(\lfloor n(\log_2(n) + a\sqrt{\log_2(n)})/2 \rfloor)] \to  2\Phi(-a),
\]
and
\[
\sup_{x_0\in \bbp_n} \ee[(\oT_{x_0}(\lfloor n(\log_2(n) + a\sqrt{\log_2(n)})/2 \rfloor))^2] \to  (2\Phi(-a))^2.
\]
These imply that for any $\epsilon>0$
\[\sup_{x_0\in\bbp_n}\pp[\oT_{x_0}(\lfloor n(\log_2(n) + a\sqrt{\log_2(n)})/2 \rfloor) > 2\Phi(-a)+\epsilon] \to 0,\]
and the conclusion follows.
\end{proof}

\subsection{An exact result for finite termination}
One can also ask whether $T(k)$ eventually attains the value $0$. Of course, this is trivially true if the initial vector is zero. But in general this is  impossible unless $n$ is a power of $2$. 
\begin{prop}
Suppose that $n$ is not a power of $2$. Then there is a vector $x_0$ such that if $x_k$ is defined as above, then $x_k\ne 0$ for all $k$. 
\end{prop}
\begin{proof}
Let $x_0 = (1 - \frac{1}{n},-\frac{1}{n},\ldots,-\frac{1}{n})$. We claim that for any $k$ and any $i$, $x_{k,i}$ equals $m/2^l - 1/n$ for some nonnegative integers $m$ and $l$ where $m$ is odd or zero. This is true for $k=0$ by definition. Suppose that this holds for some $k$. To produce $x_{k+1}$, suppose that we choose two coordinates $i$ and $j$. Suppose that $x_{k,i} = m/2^l - 1/n$ and $x_{k,j} = m'/2^{l'} - 1/n$. Without loss of generality, suppose that $l\ge l'$. Then 
\begin{align*}
x_{k+1, i} = x_{k+1,j} &= \frac{1}{2}\biggl(\frac{m}{2^l} + \frac{m'}{2^{l'}}\biggr) - \frac{1}{n}\\
&= \frac{m + 2^{l-l'}m'}{2^{l+1}} - \frac{1}{n}.
\end{align*}
If $l>l'$, then $m + 2^{l-l'}m'$ is odd and our claim is proved. If $l=l'$, then the above expression reduces to $m''/2^l -1/n$, where $m'' = (m+m')/2$. Since $m'' = 2^jr$ for some $j$ and some odd $r$, this expression becomes $r/2^{l-j} - 1/n$. Note that $l-j\ge 0$, because otherwise $x_{k+1,i}$ would be greater than $1$, which is impossible because we are always averaging quantities that are in $[-1,1]$. This completes the induction step. This also completes the proof of the lemma, because a quantity like $m/2^l - 1/n$, where $m$ is odd, cannot be zero unless $n$ is a power of $2$.
\end{proof}

On the other hand, if $n$ is a power of $2$, then $x_k$ eventually becomes zero for any starting state.
\begin{prop}
If $n$ is a power of $2$, then with probability $1$, $x_k=0$ for all large enough  $k$.
\end{prop}
\begin{proof}
If $n$ is a power of $2$, then it is easy to see that there is a particular sequence of steps that produces the vector of all $0$'s starting from any $x_0$. For example, consider the case $n=4$. We can first average coordinates $1$ and $2$, and then $3$ and $4$, and then $1$ and $3$ and then $1$ and $4$, which will render all coordinates equal and hence zero. The scheme for $n$ equal to a general power of $2$ is similar: we do averages so that coordinates in successive blocks of size $2^l$ become equal, for $l=1,2,\ldots$, until all coordinates become equal. Note that this scheme has nothing to do with the initial state.

Since the above scheme has a fixed number of steps, it occurs sooner or later as we go along. Thus, sooner or later, $x_k=0$. 
\end{proof}

\end{document}